\documentclass[12pt]{article}
\usepackage{amsmath}
\usepackage{amsfonts}
\usepackage{amsthm}
\usepackage{graphicx}
\graphicspath{{Figures/}} 
\usepackage{overpic}
\usepackage{amssymb} 
\usepackage{pictex}
\usepackage{rotating}  
\usepackage{color}
\usepackage{cite} %pour liste de ref. remplac. par un tiret

\usepackage{enumitem} %si on veut enlever l'indent dans env. enumerate

\usepackage{mathabx} %widecheck
\usepackage{etex} %pour pb de \dimen
\usepackage{comment}
\usepackage{tikz}
\usetikzlibrary{hobby} 

\numberwithin{equation}{section}

\newtheorem{theorem}{Theorem}[section]
\newtheorem{proposition}[theorem]{Proposition}
\newtheorem{lemma}[theorem]{Lemma}
\newtheorem{corollary}[theorem]{Corollary}
\newtheorem{Definition}[theorem]{Definition}

\newtheorem{Remark}[theorem]{Remark}
\newenvironment{remark}{\begin{Remark}\rm}{\end{Remark}}
\newtheorem{RHproblem}[theorem]{RH problem}

\newtheorem{Example}[theorem]{Example}

\newcommand{\C}{\mathbb{C}}
\newcommand{\D}{\mathbb D}

\newcommand{\Q}{\mathbb{Q}}
\newcommand{\R}{\mathbb{R}}
\newcommand{\T}{\mathbb{T}}
\newcommand{\Z}{\mathbb{Z}}

\newcommand{\BB}{\mathcal B}
\newcommand{\CC}{\mathcal C}
\newcommand{\DD}{\mathcal D}
\newcommand{\EE}{\mathcal E}
\newcommand{\FF}{\mathcal F}
\newcommand{\GG}{\mathcal G}

\newcommand{\PP}{\mathcal P}

\newcommand{\eps}{\epsilon}
\newcommand{\p}{\partial}
\newcommand{\vphi}{\varphi}

\renewcommand{\Im}{{\rm Im} \,}

\def\supp{\mathop{\mathrm{supp}}\nolimits}

\renewcommand{\bar}{\overline}
\renewcommand{\tilde}{\widetilde}
\renewcommand{\hat}{\widehat}
\renewcommand{\check}{\widecheck}

\usepackage[bookmarksopen, naturalnames] {hyperref}

\begin{document}
\title{%An application of 
Boundary value problems and\\ Heisenberg uniqueness pairs}
% for convex surfaces}
\author{S. Rigat and F. Wielonsky}

\maketitle 

\begin{abstract}  
We describe a general method for constructing Heisenberg uniqueness pairs $(\Gamma,\Lambda)$ in the euclidean space $\R^{n}$ based on the study of boundary value problems for partial differential equations. As a result, we show, for instance, that any pair made of the boundary $\Gamma$ of a bounded convex set $\Omega$ and a sphere $\Lambda$ is an Heisenberg uniqueness pair if and only if the square of the radius of $\Lambda$ is not an eigenvalue of the Laplacian on $\Omega$. The main ingredients for the proofs are the Paley-Wiener theorem, the uniqueness of a solution to a homogeneous Dirichlet or initial boundary value problem, the continuity of single layer potentials, and some complex analysis in $\C^{n}$. Denjoy's theorem on topological conjugacy of circle diffeomorphisms with irrational rotation numbers is also useful.
\end{abstract}
%%%%%%%%%%%%%
\section{Introduction}
The interpretation of the uncertainty principle in mathematics is a classical theme, especially in harmonic analysis, see e.g.\ \cite{HJ} for a thorough presentation of the subject. 
Heisenberg uniqueness pair is a related notion which was introduced in \cite{HM}. It consists in the following.
\\[.5\baselineskip]
A pair 
$(\Gamma,\Lambda)$, where $\Gamma$ is a hypersurface in $\R^{n}$, $n\geq2$, and $\Lambda$ is a set in $\R^{n}$, is an {\it Heisenberg uniqueness pair} (sometimes abbreviated HUP) if for any complex-valued function $g$ integrable on $\Gamma$ with respect to $d\sigma$, the surface measure on $\Gamma$, the following implication holds true,
$$\hat g(\lambda)=\int_{\Gamma}e^{-i\lambda\cdot x}g(x)d\sigma(x)=0\text{  on }\Lambda\quad\implies
\quad g=0~\text{ a.e.\ on }\Gamma.$$
%vanishes on $\Lambda$, must be the zero function on $\Gamma$.
%\\[.5\baselineskip]
The function $\hat g(\lambda)$ is the Fourier transform, in the sense of distribution, of the measure $gd\sigma$, also called the Fourier-Stieltjes transform of $g$. Note that, by the Hahn-Banach theorem, the above condition is equivalent to the fact that the vector space spanned by the family of exponentials $e^{i\lambda\xi}$, $\lambda\in\Lambda$, is weak-* dense in $L^{\infty}(\Gamma)$.
Also, a similar notion is that of a mutually annihilating pair $(S,\Sigma)$ of Borel subsets of $\R$ of positive measure, satisfying
$$
\forall\vphi\in L^{2}(S),~\supp(\hat\vphi)\subset\Sigma\quad\implies\quad\vphi=0,
$$
see \cite{HJ}.

It is easy to check that the notion of an Heisenberg uniqueness pair satisfies two elementary properties, namely,
\\[.5\baselineskip]
1) it is invariant by translation: if $(\Gamma,\Lambda)$ is an HUP then $(\Gamma+x_{0},\Lambda+\lambda_{0})$ with $x_{0},\lambda_{0}\in\R^{n}$, is also an HUP, 
\\[.5\baselineskip]
2) if $T$ is an invertible linear transformation with adjoint $T^{*}$, then 
\begin{equation}\label{rem-T}
(\Gamma,\Lambda)\text{ is an HUP}\implies(T(\Gamma),(T^{-1})^{*}(\Lambda))\text{ is also an HUP.}
\end{equation}
In \cite{HM}, it was proven that the pair made of the hyperbola and a cross lattice,
$$
\Gamma=\left\{\left(x_{1}, x_{2}\right) \in \mathbb{R}^{2} : x_{1} x_{2}=1\right\},\quad
\Lambda_{\alpha, \beta} :=(\alpha \mathbb{Z} \times\{0\}) \cup(\{0\} \times \beta \mathbb{Z}),\quad
\alpha,\beta>0,
$$
is a HUP if and only if $\alpha\beta\leq1$.
Since the publication of \cite{HM}, many other examples have appeared in the literature. For instance, in \cite{SJ1,L,SJ2,BB,GS,BA,JK}, one may find examples of HUP's in the plane $\R^{2}$ involving sets like circles, union of lines, ellipses, parabola, polygons or very specific curves. In \cite{GV,SR}, one may also find examples of HUP's in $\R^{n}$, $n>2$, again involving spheres, paraboloids or certain cones.

The aim of the present paper is to obtain Heisenberg uniqueness pairs in $\R^{n}$, $n\geq2$, under rather general assumptions on the set $\Gamma$ since, in some of our examples, it is only assumed to be the boundary of a bounded convex set.
Our method is inspired by the study in \cite{A}, which was concerned with providing rigorous foundations of the Fokas method for boundary value problems, showing, for linear elliptic pde's on bounded convex domains, that a solution to the global relation between the Dirichlet and Neumann boundary data implies the existence of a solution to the Dirichlet problem, and that the solution to the global relation is unique. The link between \cite{A} and the problem of determining Heisenberg uniqueness pairs does not seem to have been previously noticed.

When convenient, we will denote by $X$ a $n$-tuple of variables $X_{1},\ldots,X_{n}$ and by $D$ the tuple of partial derivations $\p_{X_{1}},\ldots,\p_{X_{n}}$.
For a polynomial $P(X)=\sum_{|\alpha|\leq m}a_{\alpha}X^{\alpha}\in\C[X]$, the differential operator $P(-iD)$ is defined by
$$
P(-iD)=\sum_{|\alpha|\leq m}a_{\alpha}(-i)^{|\alpha|}\p_{X_{1}}^{\alpha_{1}}\dots\p_{X_{n}}^{\alpha_{n}}.
$$
We will also denote by $Z(P)$ the subset of $\C^{n}$ of zeros of $P$, and by $Z_{\R}(P)=Z(P)\cap\R^{n}$ the subset of its real zeros.

Let $E_{P}(x)$ be a fundamental solution of the differential operator $P(-iD)$, that is 
$P(-iD)E_{P}=\delta$ in the sense of distribution, where $\delta$ denotes the Dirac delta distribution. 
We will make use of potentials of a single layer,
\begin{equation}\label{def-Phi}
\Phi g(x):=\int_{\Gamma}g(y)E_{P}(x-y)d\sigma(y),\qquad x\in\R^{n},
\end{equation}
where $g$ is some function defined on $\Gamma$, see Section \ref{Prelim} for more details.

Let us now state our main result.
\begin{theorem}\label{main}
Let $\Omega$ be a bounded convex domain in $\R^{n}$ with a piecewise $C^{1}$ boundary $\Gamma$, %of class $C^{1}$, 
and let $d\sigma$ be the surface measure on $\Gamma$ (the arc measure when $n=2$).
Let $P(X_{1},\ldots,X_{n})$ be a polynomial with real coefficients such that
\\[.5\baselineskip]
(i) $P$ is a square-free polynomial in $\C[X]$, i.e.\ each irreducible factor of $P$ is of multiplicity~1, and for at least one variable, e.g.\ $X_{1}$, one has $P(X)=a_{1}X_{1}^{d}+\cdots$, where the nonzero leading coefficient $a_{1}$ is a constant (i.e.\ independent from the other variables).
\\[.5\baselineskip]
(ii) %the variety $Z(P)\subset\C^{n}$ is defined by real equations, contains a point in 
each irreducible component $V_{i}$ of the variety $Z(P)$ is defined by a real polynomial, and contains a point in $V_{i,\R}=V_{i}\cap\R^{n}$, which is smooth as a point of $V_{i}$ (i.e.\ the Jacobian has rank 1 at this point).
\\[.5\baselineskip]
(iii) For some space $\DD_{P}(\bar\Omega)$ of functions defined on $\bar\Omega$, 
the homogeneous Dirichlet problem %for $u\in C^{2}(\bar\Omega)$,
\begin{align*}
P(-iD)u(x) & =0,\qquad x\in\Omega, %\label{DP1}
\\[5pt]
u(x) & = 0,\qquad x\in\Gamma, %\label{DP2}
\end{align*}
%where $g\in L^{1}(\Gamma)$, 
has the unique solution $u=0$ in $\DD_{P}(\bar\Omega)$.
\\[.5\baselineskip]
(iv) For some space $\GG(\Gamma)\subset L^{1}(\Gamma)$, of complex-valued functions $g$ defined on $\Gamma$, the potential $\Phi g$, associated to $g\in\GG(\Gamma)$, 
defined in (\ref{def-Phi}), is a continuous function in $\R^{n}$, and %satisfies
$$
\forall g\in\GG(\Gamma),\quad\Phi g|_{\bar\Omega}\in\DD_{P}(\bar\Omega).
$$
Then %with $Z_{\R}(P)$ the set of real zeros of $P$, 
$(\Gamma,Z_{\R}(P))$ is a Heisenberg uniqueness pair for $\GG(\Gamma)$, that is,
%$L^{p}(\Gamma)$, $p>1+\nu/(n-1-\nu)$, that is
$$
\forall g\in \GG(\Gamma),\quad\hat g=0\text{ on }Z_{\R}(P)\quad\implies\quad g=0\text{ a.e.\ on }\Gamma.
$$
\end{theorem}
%Note that the main assumption on the polynomial $P$ is that the
%homogeneous Dirichlet problem for the associated differential operator $P(-iD)$ in $\bar\Omega$ has the unique trivial zero solution.
%\\[1\baselineskip]
%1) Choice of an operator $P$ for which Dirichlet problem has a unique solution\\
% (ordre 2, constant coeff., elliptic ?)\\
%2) Paire de Lax, Relation globale\\
%3) Solution fondam. $E_{p}$, \\
%4) $U/P$ entiere, Croissance de $U/P$\\
%5) Plemelj for integral with kernel $E_{p}$ and derivative\\
%6) Heisenberg (voir le cas du disque)\\
%7) Ppe fondamental, Fisher
%%%%%%%%%%%%%%%%%%%%%
The spaces of functions $\DD_{P}(\bar\Omega)$ and $\GG(\Gamma)$ in the above statement will be chosen in accordance with the domain $\Omega$ and the polynomial $P$, see Section \ref{Appli} for details.

Applying the Theorem \ref{main} in specific cases, we derive general Heisenberg uniqueness pairs as described next. %in the next results.
\begin{theorem}\label{HUP1}
The following holds true: 
\\[.5\baselineskip]
i) Let $\Gamma$ be the boundary of any bounded convex domain in $\R^{n}$ of class $C^{1}$. Let $S_{n-1}(c_{1})$ be the sphere of dimension $n-1$ and radius $c_{1}$. Then $(\Gamma,S_{n-1}(c_{1}))$ is a HUP for $L^{p}(\Gamma)$, $p>n-1$, if and only if $c_{1}^{2}$ is not an eigenvalue of $-\Delta$ on $\Omega$.  %Heisenberg uniqueness pair 
\\[.5\baselineskip]
ii) Let $\Gamma$ be the rectangle with vertices $(0,0), (T,0), (T,L), (0,L)$, $T,L>0$, and let $\PP$ be the parabola $X_{1}+X_{2}^{2}=0$ in $\R^{2}$.
The pair $(\Gamma, \PP)$ is a HUP
%Heisenberg uniqueness pair 
for $C^{2}(\Gamma)$.
%where $\Gamma $ denotes the boundary of the rectangular domain $\Omega$, see (\ref{rect})
\\[.5\baselineskip]
iii) Let $\Gamma$ be the rectangle as in ii), and assume that $T/L\notin\Q$. Let $\Delta_{+}$ and $\Delta_{-}$ be the two lines 
$\{X_{1}= X_{2}\}$, $\{X_{1}=-X_{2}\}$ in $\R^{2}$.
Then the pair 
$(\Gamma, \Delta_{+}\cup\Delta_{-})$ is a HUP
%Heisenberg uniqueness pair  
for $C^{1}(\Gamma)$.
\\[.5\baselineskip]
iv) Let $\T$ denote the unit circle, and let $\Delta_{1}, \Delta_{2}$ be two lines 
through the origin. 
Then the pair $(\T,\Delta_{1}\cup\Delta_{2})$
is a HUP for $L^{1}(\T)$ if and only if $\Delta_{1}$ and $\Delta_{2}$ make an angle of $\pi\rho$, with $\rho$ an irrational number.
\end{theorem}
We note that the HUP in item iv) was obtained previously, in \cite{SJ1}, by a different method. Also, making use of the property (\ref{rem-T}), item iii) may be rephrased as the fact that the pair $(\CC,\Delta_{\rho}\cup\Delta_{-\rho})$ is a HUP, where $\CC$ is the rectangle of sides $\pi$ and 1, and 
$\Delta_{\pm\rho}=\{X_{2}=\pm\pi\rho X_{1}\}$, with $\rho\not\in\Q$. This shows some relation between iii) and iv).

Acually, item iv) can be extended in the following way.
\begin{theorem}\label{HUP2}
Let $\Gamma$ be any strictly convex planar curve of class $C^{2}$. There exists an angle $\rho$ between two lines $\Delta_{1}, \Delta_{2}$ through the origin, such that the pair 
$(\Gamma,\Delta_{1}\cup\Delta_{2})$ is a HUP for $L^{1}(\Gamma)$. 
\end{theorem}
Unfortunately, we were unable to characterize the angles $\rho$, in terms of $\Gamma$, for which the above pair is a HUP.
It may possibly be an interesting problem.

In Section \ref{Prelim} we recall or state several preliminary results that will be useful for our study. In Section \ref{Proof}, we give the proof of Theorem \ref{main}. In Section \ref{Appli}, we apply our result in the cases of an elliptic equation (the Helmholtz equation), the Schr\"odinger equation, a hyperbolic equation (the wave equation), and finally, transport equations. For each of these equations, we derive the corresponding Heisenberg uniqueness pairs.
%%%%%%%%%%%%%%
\section{Preliminaries}\label{Prelim}
We start with recalling a few details about the notion of a potential of a single layer, as defined in (\ref{def-Phi}). The following can be found in \cite[Chap.\ 2, \S 3]{DL}, in the case of the Newton potential.
In the sequel, $\vphi$ denotes a test function. For $g\in L^{1}(\Gamma)$, the Radon measure $gd\sigma$ is the distribution, supported on $\Gamma$,
defined by
$$
<gd\sigma,\vphi>=\int_{\Gamma}g(m)\vphi(m)d\sigma(m).
$$
The potential of a single layer $\Phi g$, associated to $g$, is then defined as the convolution of distributions,
$$
\Phi g=E*gd\sigma,
$$
where $E$ denotes a fundamental solution of the differential operator $\PP=P(-iD)$. Note that the convolution is well defined as the measure $gd\sigma$ has compact support.

By the classical result of Malgrange and Ehrenpreis,
a fundamental solution $E$ of the linear differential operator $P(-iD)$ with constant coefficients, always exists, satisfying
\begin{equation}\label{sol-fund}
P(-iD)E=\delta.
\end{equation}
With $\check E$ denoting the distribution symmetric to $E$, defined by $<\check E,\vphi>=<E,\check\vphi>$, $\check\vphi(x)=\vphi(-x)$, and from basic properties of the convolution, we get
%On a donc, avec les Th\'eor\`emes \ref{Prop-conv} et \ref{conv=fct},
\begin{equation}\label{eq-Tg}
<\Phi g,\vphi>=<gd\sigma,\check E*\vphi>=\int_{\Gamma}(\check E*\vphi)gd\sigma
=\int_{\Gamma}<\check E(y),\vphi(x-y)>g(x)d\sigma(x).
\end{equation}
Now, with $\PP^{*}$, the adjoint of the differential operator $\PP$, we have
\begin{align*}
<\PP(\Phi g),\vphi> & =<\Phi g,\PP^{*}\vphi>=\int_{\Gamma}(\check E*\PP^{*}\vphi)gd\sigma
=\int_{\Gamma}(\PP^{*}\check E*\vphi)gd\sigma
\\[5pt]
%=\int_{\Sigma}(P^{*}\check G*\vphi)gd\sigma
& =\int_{\Gamma}(\check{\PP E}*\vphi)gd\sigma=\int_{\Gamma}(\delta*\vphi)gd\sigma
=\int_{\Gamma}\vphi gd\sigma=<gd\sigma,\vphi>,
\end{align*}
and thus $\PP(\Phi g)=gd\sigma$, supported on $\Gamma$. Hence
\begin{equation}\label{P-Phi=0}
\PP(\Phi g)|_{\R^{n}\setminus\Gamma}=0.
\end{equation}
In the case when the fundamental solution $E$ is a locally integrable function, 
one gets, from (\ref{eq-Tg}),
$$
<\Phi g,\vphi>=\int_{x\in\Gamma}\int_{y\in\R^{n}}\check E(x-y)\vphi(y)dyg(x)d\sigma(x)
=<\int_{x\in\Gamma}E(y-x)g(x)d\sigma(x),\vphi(y)>
$$ 
i.e. $\Phi g$ is just the function
$$
\Phi g(y)=\int_{\Gamma}E(y-x)g(x)d\sigma(x).
$$

Next, we recall a version of the Paley-Wiener theorem that will be useful to us, see \cite[Theorem 7.3.1]{H}.
\begin{theorem}
[Paley-Wiener]\label{PW}
Let $K$ be a convex compact subset of $\R^{n}$ with supporting function $H$. Every entire function $u$ in $\C^{n}$ satisfying, for some positive integer $N$,  an estimate
$$
|u(\zeta)| \leqq C(1+|\zeta|)^{N} e^{H(\operatorname{lm} \zeta)}, \quad \zeta \in \mathbb{C}^{n},
$$
is the Fourier transform of a distribution with support contained in $K$.
\end{theorem}
%The next result is the injectivity of the Fourier-Stieltjes transform. Let $\MM_{+}^{b}(\R^{n})$ be the set of (positive) finite Borel measures on $\R^{n}$, and let
%$C_{b}(\R^{n},\C)$ be the space of continuous bounded functions on $\R^{n}$.
%\begin{theorem}[cf. {\cite[Theorem 23.4]{B}}]\label{FS-inj}
%The map
%$$
%\mu\mapsto\hat\mu(x)=\int e^{-i<t,x>}d\mu(t)
%$$
%from $\MM_{+}^{b}(\R^{n})$ to $C_{b}(\R^{n},\C)$ is injective.
%\end{theorem}
%We note that in the given reference, the Fourier transform of $\mu$ is defined without a minus sign in the exponent.
We will also make use of a result about analytic functions in $\C^{n}$ vanishing on the real points of a variety, see e.g.\ \cite[Theorem 5.1]{SO}. Recall that a point of an %irreducible 
affine variety $V \subset \mathbb{C}^n$ of dimension $d$, defined by a family of polynomials, is smooth if the Jacobian of the defining family has rank $n-d$ at this point. We first consider the case of an irreducible variety $V$.
\begin{proposition}\label{Sottile}
Let $V \subset \mathbb{C}^n$ be an irreducible 
affine variety of dimension $d$, defined by a family of $n-d$ polynomials $P_{1},\ldots,P_{n-d}\in\R[X]$. 
Assume $V$ has a smooth real point $a\in V_{\R}$.
%, i.e.\ the Jacobian 
%$(\p_{z_{j}}P_{i}(a))$, $i=1,\ldots,k$, $j=1,\ldots,n$, 
%of the defining family has rank $n-d$ at $a$.
Then $V_{\mathbb{R}}$ is Zariski dense in $V$, that is any function $f$ analytic in %a neighborhood of 
$V$, which vanishes on $V_{\R}$, must vanish on $V$.
\end{proposition}
\begin{proof}
%Let $I(V)=(P_{1},\ldots,P_{k})$ be the ideal of polynomials vanishing on $V$, spanned by $P_{1},\ldots,P_{k}\in\R[X]$. 
%{\color{red} 
%We may assume that in a neighborhood of $a$, the variety $V$ is defined by the vanishing of a family of $n-d$ polynomials $P_{1},\ldots,P_{n-d}\in\R[X]$.
Consider the smooth point $a\in V$, and assume, for simplicity, that the first minor $(\p_{z_{j}} P_{i}(a))_{i,j=1,\ldots,n-d}$ of the Jacobian  is non singular. 
By the implicit function theorem, there exists a small ball $\BB_{a}$ in $\C^{d}$, centered at $(a_{n-d+1},\ldots,a_{n})$, a neighborhood $W$ of $(a_{1},\ldots,a_{n})$, and $n-d$ analytic functions 
$$h_{j}(z_{n-d+1},\ldots,z_{n}),\qquad
j=1,\ldots,n-d,
$$ 
defined in $\BB_{a}$ such that,
\begin{align}
(z_{1},\ldots,z_{n})\in W\quad & \text{ and } \quad P_{i}(z_{1},\ldots,z_{n})=0,~ i=1,\ldots,n-d
\quad\iff \notag
\\[5pt]
\label{impl1}
 (z_{n-d+1},\ldots,z_{n})\in\BB_{a}\quad & \text{ and }\quad  z_{j}=h_{j}(z_{n-d+1},\ldots,z_{n}),~
j=1,\ldots,n-d.
\end{align}
Since the $P_{i}$'s have real coefficients, $a$ is also a regular point of $V_{\R}$.
%$$
%\partial_{x_{j}}P_{i}(a)=\partial_{z_{j}}P_{i}(a),\quad i=1,\ldots n-d,\quad j=1,\ldots,n.
%$$
Applying the implicit function theorem, this time in $\R^{n}$, we get, 
%Restricting the above assertion to real points, we get
\begin{align}
(x_{1},\ldots,x_{n})\in W\cap\R^{n}\quad & \text{ and } \quad P_{i}(x_{1},\ldots,x_{n})=0,~ i=1,\ldots,n-d
\quad\iff \notag
\\[5pt]
\label{impl2}
 (x_{n-d+1},\ldots,x_{n})\in\BB_{a}\cap\R^{d}\quad & \text{ and }\quad  x_{j}=h_{j}(x_{n-d+1},\ldots,x_{n}),~
j=1,\ldots,n-d.
\end{align}
%\begin{align}
%\forall (x_{n-d+1},\ldots,x_{n})\in\BB_{a,\R}:=\BB_{a}\cap\R^{d}, \quad& P_{i}(x_{1},\ldots,x_{n})=0,~ i=1,\ldots,n-d \notag
%\\[5pt]
%\label{impl2}
%\iff\quad &  x_{j}=h_{j}(x_{n-d+1},\ldots,x_{n}),~
%j=1,\ldots,n-d,
%\end{align}
where the functions $h_{j}$ in (\ref{impl1}) and in (\ref{impl2}) are actually the same (and are real-valued when restricted to $\R^{d}$). 
Now, for $f$ analytic on $V$ and vanishing on $V_{\R}$, the function
$$
g(z_{n-d+1},\ldots,z_{n})=f(h_{1}(z_{n-d+1},\ldots,z_{n}),\ldots,h_{n-d}(z_{n-d+1},\ldots,z_{n}),z_{n-d+1},\ldots,z_{n})
$$
is analytic in $\BB_{a}$ and, because of (\ref{impl2}), vanishes on $\BB_{a,\R}$. Hence, $g$ vanishes on the whole of $\BB_{a}$. Equivalently, $f$ vanishes in a neighborhood (in the topology of $V$) of $a$. Finally, applying the identity principle for holomorphic functions on the irreducible algebraic set $V$, see \cite[Chapter 1, \S 5.3]{C}, $f$ must vanish on all of $V$. 
\end{proof}
The above result extends easily to the case of a non irreducible variety.
\begin{corollary}\label{cor-sot}
Let $V\subset\C^{n}$ be an affine variety of dimension $d$, and let $V=\cup V_{i}$ its decomposition into a finite number of irreducible components $V_{i}$. Assume each $V_{i}$ is defined by real polynomials, and contains a smooth real point $a_{i}\in V_{i,\R}$. Then, as in Proposition \ref{Sottile}, $V_{\R}$ is Zariski dense in $V$.
\end{corollary}
\begin{proof}
It suffices to notice that $V_{\R}=\cup V_{i,\R}$ and to apply Proposition \ref{Sottile} to each of the components $V_{i}$.
\end{proof}
%As a final result in this section, we state a 
Another result about analytic functions in $\C^{n}$, useful to us, is the following lemma.
\begin{lemma}\label{lem-div}
Let $\Omega$ be an open connected subset of $\mathbb{C}^n$ and let $f$ be analytic in $\Omega$. Assume $P\in\mathbb{C}[z_1,\ldots,z_n]$ is a square-free polynomial.
% whose irreducible factors are all of multiplicity one. 
If $f$ vanishes on the zero set $Z(P)$ of $P$, then $f/P$ is analytic in $\Omega$.
\end{lemma}
\begin{proof}
By the local analytic nullstellensatz, see \cite[Chapter 3]{GR}, for each $z\in\Omega$, there exists a neighborhood $U_{z}$ of $z$ and a function $g_{z}$ analytic in $U_{z}$ such that $f=g_{z}P$. By the identity principle applied in $\Omega$, we may glue together the functions $g_{z}$ to get a function $g$ analytic in $\Omega$ such that $f=gP$ there.
%We could also directly apply the global complex-analytic version of the Hilbert nullstellensatz (cf.\ \cite{S}).
\end{proof}
Finally, for one of our applications, we will need some results about diffeomorphisms of the unit circle $\T=\R/2\pi\Z$.  The rotation number of an orientation-preserving homeomorphism $f:\T\to\T$ is defined as the  limit 
$$
\tau(f)=\lim_{n\to\infty}\frac{\tilde f^{n}(x)-x}{2\pi n},\qquad\text{mod }1,$$
independent of $x$, where $\tilde f^{n}$ denotes the $n$-th iterate of a lift $\tilde f$ to $\R$ of $f$,
satisfying $\Pi\circ\tilde  f=f\circ\Pi$, with $\Pi$ the projection $\Pi:\R\to\T$. Intuitively, the rotation number represents the
average portion of the circle that a point is moved by $f$.
\begin{lemma}[{\cite[Proposition 11.1.6]{KH}}]\label{cont-rot}
The map $f\mapsto \tau(f)$ is continuous in the uniform topology of $C(\T)$, the space of continuous maps from $\T$ to $\T$.
\end{lemma}
A basic result concerning the diffeomorphisms of the circle is the theorem of Denjoy.
\begin{theorem}[Denjoy, {\cite[Theorem 12.1.1]{KH}}] \label{Denjoy}
A $C^{1}$ diffeomorphism $f:\T\to\T$ with irrational rotation number $\tau(f)$ and derivative of bounded variation (e.g.\ a $C^{2}$ diffeomorphism) is topologically conjugate to the rotation $R_{\tau(f)}$ of angle $2\pi\tau(f)$. Namely, there exists a homeomorphism $h$ of the circle such that
$$
f=h^{-1}\circ R_{\tau(f)}\circ h.
$$
\end{theorem}
%%%%%%%%%%%%%%%%%%%%%
\section{Proof of Theorem \ref{main}}\label{Proof}
%Heisenberg pairs from Ashton method}
%Let $P\in\C[X_{1},\ldots,X_{n}]$ so that uniqueness of a solution to the Dirichlet problem for the operator $P(D)$ holds true. 
%In this section, we give a proof of Theorem \ref{main}.
\begin{proof}[Proof of Theorem \ref{main}]
Let $g\in\GG(\Gamma)$, the function space introduced in assumption (iv), such that its Fourier-Stieltjes transform $\hat g$ satisfies
$$\hat g(\lambda)=\int_{\Gamma}e^{-i\lambda\cdot x}g(x)d\sigma(x)=0,\quad\lambda\in Z_{\R}(P),$$
with $d\sigma$ the surface measure on $\Gamma$. 
Taking the Fourier transform $\FF$ in $\R^{n}$ on both sides of (\ref{sol-fund}), we get
$P(\lambda)\FF(E)=1$.
%$$
%\FF(P(-iD_{x})E(x-y))=P(\lambda)\FF(E(x-y))=e^{-i\lambda\cdot y}.
%$$ 
Moreover, The function $\hat g$ is an entire function in $\C^{n}$ which, by assumption, vanishes on $Z_{\R}(P)$. Hence, by Proposition \ref{Sottile} and its Corollary \ref{cor-sot}, it vanishes on $Z(P)$,
which implies, together with Lemma \ref{lem-div}, that the quotient $\hat g/P$ is also 
an entire function
(recall assumption (i) that $P$ is a square-free polynomial in $\C[X]$).
Hence, for the Fourier transform of $\Phi g$, we have
\begin{equation}\label{Four-trans}
\FF(\Phi g)=\FF(E* gd\sigma)=\FF(E)\FF(gd\sigma)=\hat g\FF(E)=(\hat g/P)P\FF(E)=\hat g/P.
\end{equation}
%, and finally concluding with Hartogs theorem.
Next, we use the elementary fact that, if $h(z)$ is an analytic function in a neighborhood of the closed unit disk $\bar\D\subset\C$ and $p(z)$ is a polynomial with leading coefficient $\alpha$, one has
$$
|\alpha h(0)|\leq\sup_{|z|=1}|h(z)p(z)|.
$$
Thus, applying this inequality with the variable $z_{1}$, $z=(z_{1},0,\ldots,0)$, $p(z_{1})=P(z+\lambda)$, 
$h(z_{1})=\FF(\Phi g)(z+\lambda)$, and recalling the second part of assumption (i), we get
%assuming e.g.\ that $a_{11}\neq0$, we get, with $z=(z_{1},0,\ldots,0)$,
\begin{align*}
|\FF(\Phi g)(\lambda)| & \leq a_{1}^{-1}\sup_{|z_{1}|=1}|\FF(\Phi g)(z+\lambda)P(z+\lambda)|
\leq a_{1}^{-1}\sup_{|z_{1}|=1}\left|\int_{\Gamma}e^{-i(z+\lambda)\cdot y}g(y)d\sigma(y)\right|
\\[10pt]
& \leq C\sup_{y\in\Gamma}e^{\Im(\lambda)\cdot y}=Ce^{H_{\bar\Omega}(\Im(\lambda))},
\end{align*}
%Let $(a_{1},\ldots,a_{n})\in\C^{n}$ such that $P(a_{1},\ldots,a_{n})\neq0$, and denote by $a\neq0$ the leading coefficient of the one-variable polynomial $P(z_{1}+a_{1},\ldots,a_{n})$. Then, with
where $C$ is some constant independent of $\lambda$, and $H_{\bar\Omega}$ is the supporting function of $\bar\Omega$, the closure of $\Omega$. Now, applying
the Paley-Wiener theorem, see Theorem \ref{PW}, one derives that $\Phi g$ vanishes outside $\bar\Omega$. On the other hand, from the continuity of $\Phi g$ in $\R^{n}$, recall assumption (iv), $\Phi g$ also vanishes on $\Gamma$.
%and Theorem \ref{single-layer-cont}, the integral transform (\ref{Slayer}) associated to $E_{p}(x,y)$ has no jump across $\Gamma$.
%see \cite[p.203]{M}. 
%Hence $u_{+}(x)=u_{-}(x)=0$ for $x\in\Gamma$. 
From assumption (iii), the homogeneous Dirichlet problem for the operator $P(-iD)$ has the unique zero solution in $\DD_{P}(\bar\Omega)$.  Hence $\Phi g$, which lies in $\DD_{P}(\bar\Omega)$ by assumption (iv), and satisfies, see (\ref{P-Phi=0}),
$$P(-iD)(\Phi g)(x)=0,\quad x\in\R^{n}\setminus\Gamma,$$ 
must vanish inside $\Omega$ and thus everywhere in $\R^{n}$. Consequently, 
%for $\lambda\in\R^{n}$, 
$\FF(\Phi g)=0$ i.e.\ $\hat g=\FF(gd\sigma)=0$. The measure $gd\sigma$ is a distribution with compact support, hence a tempered distribution. From the fact that the Fourier transform is an isomorphism from the space of tempered distributions to itself, we thus obtain that $g=0$.
%on $\C^{n}$, in particular on $\R^{n}$, which implies $g=0$ by injectivity of the Fourier-Stieltjes transform, see Theorem \ref{FS-inj}. 
We have thus proved that $(\Gamma,Z_{\R}(P))$ is a Heisenberg pair for the space $\GG(\Gamma)$.
\end{proof}
%\noindent
%{\bf Exemple} : Let $P(X)\in\C[X_{1},X_{2}]$ and $\Omega\subset\R^{2}$. The assumption on $g$ is
%$$
%\forall(\lambda_{1},\lambda_{2})\in Z_{\C}(P),\quad\int_{\Gamma} e^{-i(\lambda_{1}x_{1}+\lambda_{2} x_{2})}g(x)d\sigma(x)=0.
%$$
%Check directly that $g=0$ in a simple case ?
%%%%%%%%%%%%%%%%%%%%
\section{Applications}\label{Appli}
We shall restrict ourselves to quadratic polynomial 
$P(X_{1},\ldots,X_{n})\in\R[X]$ of the form
\begin{equation}\label{def-P}
P(X_{1},\ldots,X_{n})=-\sum_{i,j=1}^{n} a_{ij}X_{i}X_{j}+\sum_{i=1}^{n}b_{i}X_{i}+c,
\qquad a_{i,j}, b_{i}, c\in\R,
\end{equation}
with $(a_{ij})_{i,j}$, a symmetric matrix.
%%%%%%%%%%%%%%%%%%
\subsection{Elliptic equations}
In this section, we consider elliptic operators, that is, we assume that the matrix $(a_{ij})_{i,j}$ in (\ref{def-P}) is positive.
By a linear change of variables with real coefficients (so that real points of $Z(P)$ correspond by the change of variables), we may restrict ourselves to operators of the form
$$
P(-iD)=\Delta-i\sum_{i=1}^{n}b_{i}\p_{x_{i}}+c\quad\text{where}\quad P(X)= -\sum_{i=1}^{n} X_{i}^{2}+\sum_{i=1}^{n}b_{i}X_{i}+c,\quad b_{i},c\in\R.
$$
%Since we need that $Z(P)$ is defined by a real equation, we 
Assume that the $b_{i}$'s are zero. Then, $Z(P)$ is defined by the equation
%Note that $Z_{\R}(P)$ is defined by the two equations
\begin{equation}\label{def-Z-ell}
\sum_{i=1}^{n} X_{i}^{2}-c=0.%\qquad\sum_{i=1}^{n}b_{i}X_{i}=0.
\end{equation}
-- If $c<0$, $Z_{\R}(P)$ is the empty set so that Theorem \ref{main} does not apply.\\
-- If $c=0$ then $Z_{\R}(P)$ reduces to the singleton $0$, which is a singular point of $Z(P)$. 
%This unique point of $Z_{\R}(P)$ is a smooth point of $Z(P)$ if and only if at least one of the $b_{i}$'s does not vanish, which we assume to be the case. 
%Next, the polynomial $P(X)$ is irreducible in $\C[X]$ because
%$$
%P(x_{1},\ldots,x_{n})=Q\left(x_{1}-i\frac{b_{1}}{2},\ldots,x_{n}-i\frac{b_{n}}{2}\right),
%\qquad Q(X)= -\sum_{i=1}^{n} X_{i}^{2}-\sum_{i=1}^{n}\frac{b_{i}^{2}}{4}
%$$
%and a polynomial of the form $\sum_{i=1}^{n} X_{i}^{2}+c$, $c\in\C$, is irreducible in $\C[X]$. Finally
%$$
%-\left(\Delta+\sum_{i=1}^{n}b_{i}\p_{x_{i}}\right)
%$$
%is a positive operator which implies that the Dirichlet problem (\ref{DP1})-(\ref{DP2}) on $\Omega$ has the unique solution $u=0$. Applying Theorem \ref{main}, we may thus derive that $(\Gamma,\{0\})$ is an Heisenberg pair ({\color{red} Ca me parait faux mais ou est l'erreur ?}).
\\
-- If $c=c_{1}^{2}>0$, with $c_{1}>0$, then $Z_{\R}(P)$ is the sphere $S_{n-1}(c_{1})$ of dimension $n-1$ and radius $c_{1}$.
%, whose points are non singular points of $Z(P)$. 

Before we elaborate on the third case, and prove the first item in Theorem \ref{HUP1}, let us mention a result about continuity of
%general integral transforms, or 
potentials of a single layer of the form
\begin{equation}\label{def-single}
(\Phi g)(x)=\int_{\Gamma}g(y)\vphi(x-y)d\sigma(y),\qquad x\in\R^{n}.
\end{equation}
%The result we will need is the following one.
\begin{theorem}\label{single-layer-cont}
Assume the following holds :\\
1) the hypersurface $\Gamma$ is of class $C^{1}$, \\
2) The function $\vphi:\R^{n}\setminus\{0\}\to\R$ is continuous,\\
3) The function $\vphi$ is weakly singular at 0, namely there exists constants $0<\nu<n-1$ and $C>0$ such that
\begin{equation*}%\label{weak-sing}
|\vphi(x)|\leq\frac{C}{|x|^{\nu}}\quad\text{as }x\to0.
\end{equation*}
4) $g\in L^{p}_{d\sigma}(\Gamma)$ with $p>1+\nu/(n-1-\nu)$.
\\
Then the potential $\Phi g$ in (\ref{def-single}) is continuous on $\R^{n}$. In particular, its limits on $\Gamma$ from the inside and the outside coincide,
$$
(\Phi g)_{+}(x)=(\Phi g)_{-}(x),\qquad x\in\Gamma.
$$
\end{theorem}
Since we were unable to find a convenient reference for this result, we provide a proof of Theorem \ref{single-layer-cont} in the Appendix.
\begin{proof}[Proof of i) in Theorem \ref{HUP1}]
We assume $c=c_{1}^{2}>0$.
When $n\geq2$ and $c\neq0$, the polynomial in (\ref{def-Z-ell}) is irreducible in $\C[X]$, and all points of $Z(P)$ are non singular. Moreover, the homogeneous Dirichlet problem for the Helmholtz operator $\Delta+c_{1}^{2}$ on $\Omega$ has non trivial solutions in $C^{0}(\bar\Omega)\cap C^{2}(\Omega)$ only when $c_{1}^{2}$ is an
an eigenvalue %$\lambda_{0},\ldots,\lambda_{k},\ldots\to\infty$, 
of the positive operator $-\Delta$ on $\Omega$. Finally, a fundamental solution of the Helmholtz operator is given (see e.g.\ \cite[p.40]{OW}), up to a multiplicative constant, by
$$
E_{n}(x) 
%:=\begin{cases}
%{-\frac{i}{4} H_{0}^{(1)}(\omega|x|),} & {d=2} \\[5pt] 
%{\displaystyle -\frac{e^{i \omega|x|}}{4 \pi|x|},} & {d=3}
%\end{cases}
=|x|^{1-n/2}Y_{n/2-1}(c_{1}|x|),
$$
for \(x \neq 0,\) where $Y_{\alpha}$ is the Bessel function of the second kind of order $\alpha$.
Near 0, the following estimates hold,
$$
Y_{\alpha}(z)\sim
\begin{cases}
(2/\pi)\log(z),\quad& \alpha=0,
\\[5pt]
-(\Gamma(\alpha)/\pi)(2/z)^{\alpha},\quad & \alpha>0.
\end{cases}
$$
and thus
\begin{equation}\label{sol-fonda-Helm}
E_{n}(x)\sim
\begin{cases}
(2/\pi)\log(|x|),\quad& n=2,
\\[5pt]
-(\Gamma(n/2-1)/\pi)(2/c_{1})^{n/2-1}|x|^{2-n},\quad & n>2.
\end{cases}
\end{equation}
In view of (\ref{sol-fonda-Helm}), Theorem \ref{single-layer-cont} applies with $\vphi=E_{n}$ and $p>n-1$. Together with Theorem \ref{main}, we get that $(\Gamma,S_{n-1}(c_{1}))$ is a HUP for $L^{p}(\Gamma)$, $p>n-1$. 

Assume now that $c_{1}^{2}$ is an eigenvalue of $-\Delta$ on $\Omega$. Then, there exists some nonzero solution $u\in C^{2}(\Omega)\cap C(\bar\Omega)$ of the homogeneous Dirichlet problem in $\Omega$. Moreover, it is known, see e.g.\ Theorems 3.27 and 3.1 of \cite{CK} for the case $n=3$, that $u$ actually lies in $C^{1}(\bar\Omega)$, and that
$$
\int_{\Gamma}\frac{\partial u}{\partial \nu}(y) E_{n}(x- y)d \sigma(y)=
\begin{cases}
u(x), \quad  & x \in \Omega,
\\[5pt]
~0,  \quad & x \in \R^{n} \setminus \bar{\Omega},
\end{cases}
$$
where $\nu$ denotes the outward normal to the surface $\Gamma$. Hence, with $v=\p u/\p \nu\in C^{1}(\Gamma)$, one obtains $u=\Phi v$. First, notice that $v\neq0$ since $u\neq0$. Second, similarly to (\ref{Four-trans}), one has
$$
P(\lambda)\FF(u)(\lambda)=\hat v(\lambda),
$$
and $\FF(u)$ is an entire function since $u$ is continuous, with compact support, hence in $L^{1}(\R^{n})$. 
Hence, $\hat v$ vanishes on $Z_{\R}(P)$ and $(\Gamma,S_{n-1}(c_{1}))$ is not a HUP. This finishes the proof of
 item i) of Theorem \ref{HUP1}.
\end{proof}
\begin{remark}
When $\Omega$ is a ball, it is known that the eigenvalues of $\Delta$ are the zeros of Bessel functions of the first kind. 
\end{remark}
\begin{remark}
By the remark in (\ref{rem-T}) and the fact that convexity is preserved by linear transformation, one can make use of a linear diagonal transformation 
$$
T:~(x_{1},\ldots,x_{n})\mapsto (a_{1}x_{1},\ldots,a_{n}x_{n}),\quad a_{i}>0,\quad i=1,\ldots,n,$$
normalized by $\sum_{i=1}^{n}a_{i}^{2}=1$,
so that the statement i) in Theorem \ref{HUP1} becomes the following one. Let $\EE_{n-1}(c_{1})$ be the ellipsoid
$$
\sum_{i=1}^{n}\left(\frac{X_{i}}{a_{i}}\right)^{2}=c_{1}^{2}, 
$$
%normalized by $\sum_{i=1}^{n}a_{i}^{2}=n^{2}$ 
whose vector of semi-axes $c_{1}(a_{1},\ldots,a_{n})$ has norm $c_{1}$.
%so that $\EE_{n-1}(c_{1})$ has (quadratic) mean radius $c_{1}$. 
Let $\Gamma$ be the boundary of a bounded convex domain in $\R^{n}$.
Then, $(\Gamma,\EE_{n-1}(c_{1}))$ is a HUP for $L^{p}(\Gamma)$, $p>n-1$ if $c_{1}^{2}$ is different from an eigenvalue of the operator
$$
-\sum_{i=1}^{n}\frac{1}{a_{i}^{2}}\frac{\p^{2}}{\p x_{i}^{2}}\quad\text{on }\Omega.
$$
\end{remark}
%Assume first that the $b_{i}$'s are equal to zero. see e.g. \cite[Theorem 3.3]{GT}
%%%%%%%%%%%%%%
\subsection{The Schr\"odinger equation}
In this section, we prove item ii) of Theorem \ref{HUP1} by considering the 1-dimensional Schr\"odinger equation, corresponding to the operator
$$
P(-iD)=i\p_{t}+\p_{x}^{2}\quad\text{where}\quad P(X)=-X_{1}-X_{2}^{2}.
$$
Then $Z(P)=\{(X_{1},X_{2})\in\C^{2},~X_{1}+X_{2}^{2}=0\}$ and $Z_{\R}(P)$ is a parabola in $\R^{2}$. 

In the literature, a fundamental solution of the Schr\"odinger equation is usually given by
%up to a multiplicative constant, by
$$
H(t)t^{-1/2}\exp(ix^{2}/(4t)),\qquad t\geq0,
$$
see e.g.\ \cite[Sect.\ 6.2]{T}, where $H(t)$ denotes the Heaviside function
$$
H(t)=0\text{ for }t\leq0\text{ and }H(t)=1\text{ for }t>0.
$$ 
Actually, one can also take as a fundamental solution,
$$
E_{1}(t,x)=t^{-1/2}\exp(ix^{2}/(4t)),\qquad t\neq0,
$$
without the factor $H(t)$, which allows one to reconstruct solutions to the Schr\"odinger equation for $t\in(-\infty,\infty)$ from an initial data at $t=0$, see \cite[Section 4.3.1.b, Example~3]{E}.

For some $T,L>0$, and $I=(0,L)$, we consider the rectangular domain 
\begin{equation}\label{rect}
\Omega:=(0,T)\times I\subset\R^{2},
\end{equation} 
with vertices 
$$V_{1}=(0,0),\quad V_{2}=(T,0), \quad V_{3}=(T,L), \quad V_{4}=(0,L).$$ 
Let 
\begin{equation}\label{Phi-Schr}
\Phi g(t,x)=\int_{\Gamma}g(u,y)\frac{e^{i(x-y)^{2}/(4(t-u))}}{\sqrt{t-u}}d\sigma(u,y),
\end{equation}
where $g$ is a function defined on the boundary $\Gamma$ of $\Omega$, and $d\sigma$ is the arc measure on $\Gamma$, which we assume to be positively oriented.
Also, we assume that the square root in (\ref{Phi-Schr}) satisfies
$$
\sqrt{t-u}=-i\sqrt{|t-u|}\quad\text{when }t<u.
$$
Our aim is to show the following result.
\begin{proposition}\label{cont-Schr} Assume the function $g$ to be of class $C^{2}$ on $\Gamma_{1}$ and $\Gamma_{3}$.
Then the function $\Phi g(t,x)$ is of class $C^{2}$ on $\R^{2}\setminus(\Delta_{0}\cup\Delta_{T})$, where $\Delta_{0}, \Delta_{T}$ are the vertical lines $t=0$ and $t=L$ respectively. Moreover, $\Phi g(t,x)$ admits left and right limits at each point of $\Delta_{0}$, resp.\ $\Delta_{L}$, and these limits coincide.
%continuous on $\R^{2}\setminus\{V_{1},V_{2},V_{3},V_{4}\}$.
\end{proposition}
\begin{proof}
We decompose the integral according to the four sides $\Gamma_{1}$, $\Gamma_{2}$, $\Gamma_{3}$, $\Gamma_{4}$, of the rectangle, starting with the horizontal side $\Gamma_{1}$ from $(0,0)$ to $(T,0)$, and get
$$
\Phi g(t,x)=\Phi_{1}g(t,x)+\Phi_{2}g(t,x)-\Phi_{3}g(t,x)-\Phi_{4}g(t,x),
$$
with
\begin{alignat*}{3}
\Phi_{1} g(t,x) & :=  \int_{u=0}^{T}g(u,0)\frac{e^{ix^{2}/(4(t-u))}}{\sqrt{t-u}}du,
\quad&& \Phi_{2} g(t,x) && :=  \int_{y=0}^{L}g(T,y)\frac{e^{i(x-y)^{2}/(4(t-T))}}{\sqrt{t-T}}dy,
\\[10pt]
\Phi_{3} g(t,x) & :=  \int_{u=0}^{T}g(u,L)\frac{e^{i(x-L)^{2}/(4(t-u))}}{\sqrt{t-u}}du,
\quad&& \Phi_{4} g(t,x) && :=  \int_{y=0}^{L}g(0,y)\frac{e^{i(x-y)^{2}/(4t)}}{\sqrt{t}}dy.
\end{alignat*}
It is clear that the first and third integrals 
%in (\ref{Phi-Schr}) 
are well-defined for any $(t,x)\in\R^{2}$ since the integrands are absolutely integrable. The second one is well-defined when $t\neq T$, and the fourth one is well defined when $t\neq0$. Moreover, the four integrals define
$C^{\infty}$ functions with respect to the variable $x\in\R$. The second and fourth ones define $C^{\infty}$ functions with respect to the variable $t$ when $t\notin\{0,T\}$. The first and third ones are also $C^{\infty}$ functions of $t$ when $t\notin[0,T]$. When $0\leq t\leq T$, we may write, with a change of variable $v=t-u$,
$$
\Phi_{1} g(t,x)  =  \int_{v=t-T}^{t}g(t-v,0)\frac{e^{ix^{2}/(4v)}}{\sqrt{v}}dv,
\quad \Phi_{3} g(t,x)  =  \int_{v=t}^{t-T}g(t-v,L)\frac{e^{i(x-L)^{2}/(4v)}}{\sqrt{v}}dv,
$$
which shows, together with the Leibniz integral rule, that the two integrals are $C^{2}$ functions of $t$ (recall that we assumed $g(u,y)$ to be a $C^{2}$ function on $\Gamma_{1}$ and $\Gamma_{3}$). 

Finally, it remains to consider 
$\Phi_{2}g(t,x)$ and $\Phi_{4}g(t,x)$, respectively when $t=T$ and $t=0$. Obviously, the integral expressions for $\Phi_{2}g(T,x)$ and $\Phi_{4}g(0,x)$ are not defined, we thus compute the left and right limits as $t\to T_{\pm}$ or $t\to 0_{\pm}$. Recall from the method of stationary phase that, for $\alpha<0<\beta$,
$$
\frac{1}{\sqrt{\pi \epsilon}} \int_{\alpha}^{\beta} e^{iy^{2}/\eps} a(y) d y= e^{i {\pi}/{4}}a(0)+O(\epsilon) \quad \text { as } \epsilon \rightarrow 0_{+},
$$
where $a(y)$ is a $C^{1}$ function, see \cite[Theorem 13.1]{O} or \cite[Section 4.5.3]{E}.
Hence, for $0<x<L$,
$$
\lim_{t\to T_{+}}\Phi_{2}g(t,x)=\lim_{t\to T_{+}}\int_{y=0}^{L}g(T,y)
\frac{e^{i(x-y)^{2}/(4(t-T))}}{\sqrt{t-T}}dy=\frac{\sqrt{\pi}}{2}e^{i\pi/4}g(T,x),
$$
and
\begin{align*}
\lim_{t\to T_{-}} \Phi_{2}g(t,x) 
& = \lim_{t\to T_{-}}\int_{y=0}^{L}g(T,y)
\frac{e^{-i(x-y)^{2}/(4(T-t))}}{-i\sqrt{T-t}}dy%=\frac{\sqrt{\pi}}{2}e^{i\pi/4}g(T,x)
\\[10pt]
& = i\lim_{t\to T_{-}}\int_{y=0}^{L}g(T,y)
\frac{\bar{e^{i(x-y)^{2}/(4(T-t))}}}{\sqrt{T-t}}dy
=i\frac{\sqrt{\pi}}{2}e^{-i\pi/4}g(T,x)=\lim_{t\to T_{+}}\Phi_{2}g(t,x),
\end{align*}
which shows, together with the assumption that $x\in[0,L]\mapsto g(T,x)$ is continuous, that $\Phi_{2}g$ can be extended to the right vertical side $\Gamma_{2}$ as a continuous function. With help of \cite[Theorem 13.2]{O}, one could also check that, for $x\not\in(0,L)$, the left and right limits of $\Phi_{2}g(t,x)$ coincide at $T$, and are equal to 0 when $x\not\in[0,L]$. Though this will be not important for our analysis, we note that $\Phi_{2}g$ has a discontinuity at the right vertices $(T,0)$ and $(T,L)$ of the rectangle if the function $g(t,x)$ does not vanish at these points.

One may check that the function $\Phi_{4}g$ exhibits similar behavior on the line $t=0$, in particular, it can be continuously extended on the left vertical side $\Gamma_{4}$ of the rectangle.
\end{proof}
Next, in order to apply Theorem \ref{main}, we need an uniqueness result for the initial boundary value problem related to the Schr\"odinger equation, see \cite[Lemma 7.1.1, Theorem 7.2.1]{CH},
\begin{theorem}\label{IBVP-Schr}
%Assume a function $u(t,x)\in C^{0}(\bar\Omega)\cap C^{1}((0,T)\times[0,L])\cap C^{2}(\Omega)$ satisfies
The initial boundary value problem
$$
\begin{cases}
iu_{t}(t,x)+\Delta u(t,x)=0,\quad x\in I,~t>0,
\\[5pt]
u(0,x)=\vphi(x)\in H_{0}^{1}(I),\quad x\in \bar I=[0,L],
\\[5pt]
u(t,0)=0,~u(t,L)=0,\quad t\geq0,
\end{cases}
$$
%Then $u=0$ in $\bar\Omega$.
has a unique solution $u$ in the space $C([0,\infty),L^{2}(I))$
of continuous functions from $[0,\infty)$ to $L^{2}(I)$ 
(actually, the solution is in 
$$C([0,\infty),H_{0}^{1}(I))\cap C^{1}([0,\infty),H^{-1}(I)),$$
but this result will not be needed to us). Here, $H_{0}^{1}(I)$ and $H^{-1}(I)$
 denote the usual Sobolev spaces of the segment $I$.
 \end{theorem}
Since $\Phi g(t,x)$ is a $C^{2}$ function when $t>0$, that is on the right of the vertical line $\Delta_{0}$, in order to prove that $t\mapsto\Phi g(t,x)$ is a continuous function from  $[0,T)$ to $L^{2}(I)$, it is sufficient to check that
$$
\|\Phi g(t,x)-\Phi g(0,x)\|_{L^{2}(I)}\to0\quad\text{as }t\to0_{+},
$$
and, more precisely, that
\begin{equation}\label{est-phi4}
\|\Phi_{4} g(t,x)-\Phi_{4} g(0,x)\|_{L^{2}(I)}\to0\quad\text{as }t\to0_{+},
\end{equation}
since $\Phi_{i}g$, $i=1,2,3$, are of class $C^{2}$ in a neighborhood of the segment $\{0\}\times I$. The fact that (\ref{est-phi4}) holds true can be derived from uniform error estimates for the stationary phase approximations, see \cite{O2}. We let the details to the reader.

Putting together Proposition \ref{cont-Schr} and Theorem \ref{IBVP-Schr} applied with the initial data $\vphi(x)=0$ on $\bar I=[0,L]$, we see that the method from Theorem \ref{main} applies, leading to item ii) of Theorem \ref{HUP1}.
%\begin{theorem}
%The pair $(\Gamma, \PP)$ is an Heisenberg uniqueness pair for $C^{2}(\Gamma)$, where $\Gamma $ denotes the boundary of the rectangular domain $\Omega$, see (\ref{rect}), and $\PP$ denotes the parabola $X_{1}+X_{2}^{2}=0$ in $\R^{2}$.
%\end{theorem}
%%%%%%%%%%%%%%
\subsection{Hyperbolic equations}
Here we prove item iii) of Theorem \ref{HUP1} by considering the 1-dimensional wave equation, that is, the operator
$$
P(-iD)=\p_{t}^{2}-\p_{x}^{2}\quad\text{where}\quad P(X)=-X_{1}^{2}+X_{2}^{2}=(X_{2}-X_{1})(X_{2}+X_{1}).
$$
Then $Z(P)$ decomposes into the union of two irreducible varieties, namely the two lines $X_{1}=\pm X_{2}$ in $\C^{2}$, and 
$$Z_{\R}(P)=\Delta_{+}\cup\Delta_{-}\subset\R^{2},\qquad \Delta_{+}:X_{1}= X_{2},\quad\Delta_{-}: X_{1}=-X_{2}.$$
A fundamental (distributional) solution of the wave equation is known to be, up to a multiplicative constant,
$$
E_{1}(t,x)=H(t)H(t-x)H(t+x),
$$
see \cite[Sect.\ 7]{T}, where $H$ still denotes the Heaviside function.

As in the case of the Schr\"odinger equation, we consider the rectangular domain $\Omega\subset\R^{2}$ with vertices $(0,0)$, $(T,0)$, $(T,L)$, $(0,L)$, with $T,L>0$. Let $g$ be a $C^{1}$ function defined on the boundary $\Gamma$ of $\Omega$. Then,
$$
\Phi g(t,x)=\int_{\Gamma}g(u,y)H(t-u)H((t-u)-(x-y))H((t-u)+(x-y))d\sigma(u,y),
$$
where $d\sigma$ is the arc measure on $\Gamma$, which we assume to be positively oriented. In view of the definition of $H$, the function $\Phi g$ rewrites as
\begin{equation}\label{Phi-g-wave}
\Phi g(t,x)=\int_{\Gamma_{t,x}}g(u,y)d\sigma(u,y),
\end{equation}
where $\Gamma_{t,x}$ is the subarc (possibly empty) of $\Gamma$, which intersects the sector
$$
\{(u,y)\in\R^{2},~u< t,\quad u-y< t-x,\quad u+y<t+x\}.
$$
Depending on the location of the point $(t,x)\in\R^{2}$, it is readily checked that the endpoints of $\Gamma_{t,x}$ are two of the points
$$
(0,x-t),~(0,t+x),~(t-x,0),~(t+x,0),~(T,T-t+x),~(T,t+x-T),~(L+t-x,L),~(t+x-L,L).
$$
Because $g$ has been assumed to be a $C^{1}$ function, and all coordinates of the above points are just affine functions of $t$ and $x$, it follows from (\ref{Phi-g-wave}) that the function $\Phi g$ is of class $C^{2}$ in $\R^{2}$.

Assuming $\hat g=0$ on the two lines $X_{1}=\pm X_{2}$ in $\R^{2}$, we know from the proof of Theorem \ref{main} that the function $\Phi g$ vanishes outside of $\Omega$, and by continuity also on $\Gamma$. Hence, $\Phi g$, which belongs to $C^{2}(\Omega)\cap C^{1}(\bar\Omega)$ is a solution of the wave equation, and solves the homogeneous Dirichlet problem in $\bar\Omega$.
Let us now recall the following result, see \cite[Theorem 1]{DZ1},
\begin{theorem}
The homogeneous Dirichlet problem
$$(\p_{t}^{2}-\p_{x}^{2})f=0\text{ in }\Omega,\quad f=0\text{ on }\Gamma,$$
has the unique solution $f=0$ in $C^{2}(\Omega)\cap C^{1}(\bar\Omega)$ if and only if $T/L$ is an irrational number.
\end{theorem}
Actually, \cite{DZ1} considers the $(n+1)$-dimensional wave equation in a rectangular domain, and \cite{DZ2} considers more general hyperbolic equations in cylindrical domains.

From what precedes, we see that Theorem \ref{main} applies, leading to item iii) of Theorem \ref{HUP1}.
%\begin{theorem}
%With the previous notations, assume that $T/L\notin\Q$. Then the pair 
%$(\Gamma, \Delta_{+}\cup\Delta_{-})$ is an Heisenberg uniqueness pair  for $C^{1}(\Gamma)$.
%\end{theorem}
More generally, making use of (\ref{rem-T}), one gets that, for $c>0$, $(\Gamma, \Delta_{c,+}\cup\Delta_{c,-})$ is an Heisenberg uniqueness pair  for $C^{1}(\Gamma)$ if $T/(cL)\notin\Q$, where
$\Delta_{c,+}$ and $\Delta_{c,-}$ are the lines with slopes $c$ and $-c$.
%%%%%%%%%%%%%%%%%%%%%
\subsection{A simple first order equation} %with a distributional fundamental solution}
The goal of this section is to prove item iv) of Theorem \ref{HUP1} and Theorem \ref{HUP2}. 
We consider the simple differential equation of first order $u_{y}=0$ in $\R^{2}$, associated to the operator
$$
P(-iD)=i\p_{y}\quad\text{where}\quad P(X)=-X_{2},
$$
in a convex domain $\Sigma$ with boundary $\Gamma$. The set $\Lambda$ is the horizontal line $\Delta_{0}:=\{X_{2}=0\}$.

As is easily checked, a fundamental solution is given by the distribution
$$
E:\vphi\mapsto\int_{0}^{\infty}\vphi(0,s)ds.
$$
In view of (\ref{eq-Tg}), we have, for $g$ a function on $\Gamma$, integrable with respect to $d\sigma$, $\vphi$ a test function, and $x=(x_{1},x_{2})$,
$$
<\Phi g,\vphi>=\int_{\Gamma}<E(y),\vphi(x+y)>g(x)d\sigma(x)
=\int_{x\in\Gamma}\int_{y=0}^{\infty}\vphi(x_{1},x_{2}+y)g(x)dyd\sigma(x).
$$
%For the moment, w
We choose as a domain the unit disk $\Sigma=\D$ with boundary the unit circle $\Gamma=\T$. We set
$$
\T_{+}=\{(x_{1},x_{2})\in\T,~x_{2}\geq0\},\quad\T_{-}=\{(x_{1},x_{2})\in\T,~x_{2}<0\},
$$
and we parameterize these two arcs by the functions $s\in[-1,1]\to(s,\pm\sqrt{1-s^{2}})$, so that
$$
d\sigma(x)=\frac{|s|ds}{\sqrt{1-s^{2}}}.
$$
Setting $g(x)=g_{+}(s)$ for $x\in\T_{+}$ and $g(x)=g_{-}(s)$ for $x\in\T_{-}$, we get, 
%with a positive orientation of $\T$,
\begin{align*}
<\Phi g,\vphi> & =\int_{s=-1}^{1}\int_{y=0}^{\infty}\vphi(s,\sqrt{1-s^{2}}+y)|s|g_{+}(s)\frac{dyds}{\sqrt{1-s^{2}}}
\\[5pt]
& + \int_{s=-1}^{1}\int_{y=0}^{\infty}\vphi(s,-\sqrt{1-s^{2}}+y)|s|g_{-}(s)\frac{dyds}{\sqrt{1-s^{2}}}.
\end{align*}
Letting $u=y+\sqrt{1-s^{2}}$ in the first double integral and $u=y-\sqrt{1-s^{2}}$ in the second one, the above expression becomes
\begin{align*}
<\Phi g,\vphi> =\int_{(s,u)\in D_{+}}\vphi(s,u)|s|g_{+}(s)\frac{d\lambda(s,u)}{\sqrt{1-s^{2}}}
%\\[5pt]
+ \int_{(s,u)\in D_{-}}\vphi(s,u)|s|g_{-}(s)\frac{d\lambda(s,u)}{\sqrt{1-s^{2}}},
\end{align*}
where $D_{+}$ (resp.\ $D_{-}$) denotes the subset of $\R^{2}$ which lie above $\T_{+}$ (resp.\ $\T_{-}$), and $d\lambda$ denotes the planar Lebesgue measure. Thus, the distribution $\Phi g$ is actually a locally integrable function in $L^{1}_{loc}(\R^{2})$, namely
\begin{align*}
\Phi g(s,u) & =\chi_{D_{+}}(s,u)\frac{g_{+}(s)|s|}{\sqrt{1-s^{2}}}
+\chi_{D_{-}}(s,u)\frac{g_{-}(s)|s|}{\sqrt{1-s^{2}}}
\\[5pt]
& =\chi_{\D}(s,u)\frac{g_{-}(s)|s|}{\sqrt{1-s^{2}}}+
\chi_{D_{+}}(s,u)(g_{+}(s)+g_{-}(s))\frac{|s|}{\sqrt{1-s^{2}}},
\end{align*}
where $\chi_{A}$ denotes the characteristic function of a set $A\subset\R^{2}$.

Now, assume that $\hat g=0$ on $\Lambda$.
Here we can not apply directly Theorem \ref{main} since, a priori, $\Phi g$ is not a continuous function. Nevertheless, the first part of the proof of Theorem \ref{main} tells us that $\Phi g=0$ on $D_{+}$, that is
\begin{equation}\label{rel-g1}
g_{+}(s)=-g_{-}(s),\quad a.e.~s\in[-1,1]\quad\iff\quad
g(e^{i\theta})=-g(e^{-i\theta}),\quad a.e.~\theta\in[0,2\pi].
\end{equation}

Now, if we assume that, for some given angle $\rho\in(0,\pi)$, one has $\hat g=0$ on the line $\Delta_{\rho}:=\{X_{1}=(\cot\rho) X_{2}\}$, and consider, instead of $u_{y}=0$, the differential equation 
$$u_{x}-(\cot\rho) u_{y}=0,$$ 
we derive, in a completely similar fashion, that
\begin{equation}\label{rel-g2}
g(e^{i(\rho+\theta)})+g(e^{i(\rho-\theta)})=0,\quad a.e.~ \theta\in[0,2\pi].
\end{equation}
Combining (\ref{rel-g1}) and (\ref{rel-g2}), we obtain
$$
g(e^{i(\theta+\rho)})=g(e^{i(\theta-\rho)}),\quad a.e.~\theta\in[0,2\pi],
$$
that is, $2\rho$ is a period of the function $\theta\mapsto g(e^{i\theta})$. By uniqueness of the Fourier coefficients of a function in $L^{1}(\T)$, 
%With
%$$
%g(e^{i\theta})=\sum_{n\in\Z}a_{n}e^{in\theta},
%$$
%the Fourier expansion of $g$, 
we get
$$
a_{n}=a_{n}e^{i2n\rho},\quad n\in\Z.
$$
Assume $a_{n}\neq0$ for some $n\neq0$. Then $2n\rho=2k\pi$ for some $k\in\Z$ which implies that $\rho/\pi$ is a rational number. Hence, if $\rho/\pi\not\in\Q$, $g$ is a constant function, and by (\ref{rel-g1}), it has to vanish almost everywhere. Finally, making use of the property (\ref{rem-T}) with $T$ a rotation, we obtain the HUP stated in item iv) of Theorem \ref{HUP1}.
%the following result.
%\begin{theorem}
%The pair $(\T,\Delta_{0}\cup\Delta_{\rho})$, where the angle between $\Delta_{\rho}$ and $\Delta_{0}$ is $\rho$, with $\rho/\pi$ an irrational number, is a HUP for $L^{1}(\T)$.
%\end{theorem}

Conversely, assume that $\rho/\pi=k/n\in\Q$. Then, the nonzero function $g(e^{i\theta})=e^{in\theta}-e^{-in\theta}$ satisfies (\ref{rel-g1}) and (\ref{rel-g2}). Hence, with $\Phi_{1}g$ and $\Phi_{2}g$ the potentials associated to the fundamental solutions of $\p_{y}$ and $\p_{x}-(\cot\rho)\p_{y}$, one gets that $\Phi_{1}g$ and $\Phi_{2}g$ have compact support since they vanish outside $\D$. Consequently, their Fourier transforms are entire functions in $\C^{n}$, which, in view of (\ref{Four-trans}), implies that $\hat g$ vanishes on $\Delta_{0}\cup\Delta_{\rho}$. Hence, $(\T,\Delta_{0}\cup\Delta_{\rho})$ cannot be a HUP, and item iv) of Theorem \ref{HUP1} is completely proved.

The above result for the circle $\T$ can be extended, in some sense, to any strictly convex curve $\Gamma$ of class $C^{2}$, see Theorem \ref{HUP2}. A proof of this theorem is the goal of the remaining part of this section.
Thus we consider a function $g\in L^{1}(\Gamma)$ such that $\hat g$ vanishes on the union $\Delta_{0}\cup\Delta_{\rho}$ of two lines. We introduce two diffeomorphisms $s_{0}$ and $s_{\rho}$ on $\Gamma$ defined by the condition that, for $x\in\Gamma$, the segment $(x,s_{0}(x))$ is vertical (resp.\ $(x,s_{\rho}(x))$ makes an angle $\rho+\pi/2$ with the oriented horizontal axis). Note that, by the assumption of strict convexity of $\Gamma$, $s_{0}$ and $s_{\rho}$ are well-defined. They are involutive maps, and $s_{0}$ %(resp.\ $s_{\rho}$) 
has exactly two fixed points $M_{1},M_{2}$ which are the points where the tangeants to $\Gamma$ are vertical. %(resp.\ have the direction of $\rho+\pi/2$). 
Setting $t_{\rho}=s_{\rho}\circ s_{0}$, a reasoning similar to the one we just made for the circle leads to the conclusion that
\begin{equation}\label{inv-g-gam}
g(s_{0}(x))=-g(x),\qquad g(s_{\rho}(x))=-g(x),\quad a.e.\ x\in\Gamma,
\end{equation}
and thus,
\begin{equation}\label{inv-g-gamma}
g(t_{\rho}(x))=g(x),\quad a.e.\ x\in\Gamma.
\end{equation}
Next, pick any point inside $\Gamma$. By translation, we may assume it is the origin $O$ of the plane, and by dilation, we may also assune that the unit circle $\T$ lies inside $\Gamma$. We denote by $p$ the diffeomorphism from $\Gamma$ to $\T$,
$$
p:\Gamma\to\T,\qquad x\mapsto x/\|x\|,
$$
where $\|x\|$ denotes the euclidean norm in $\R^{2}$. 
\begin{center}
\begin{tikzpicture}[scale=3.]
\coordinate (a) at (.5,0);
\coordinate (b) at (2,0);
\coordinate (c) at (1.12,1.4); %point p(M)
\coordinate (d) at (1.5,1.5); %point M
\coordinate (e) at (-.5,1.5);

\coordinate (f) at (-.56,1.2); % point M2
\coordinate (g) at (2.15,.4); % point M1
\coordinate (h) at (.73,1.3); %centre du cercle

\coordinate (i) at (1.5,-.26); %point sous M
\path[draw,use Hobby shortcut,closed=true]
(a)..(b)..(d)..(e);
%\draw (a) node{$\bullet$};
%\draw (b) node{$\bullet$};
\draw (c) node{$\bullet$};
\draw (1.33,1.35) node {$p(M)$};
\draw (d) node{$\bullet$};
\draw (1.6,1.6) node {$M$};

\draw (f) node{$\bullet$};
\draw (-1.05,1.2) node {$M_{2}=t_{\rho_{1}}(M_{1})$};
\draw (g) node{$\bullet$};
\draw (2.6,.47) node {$M_{1}=t_{\rho_{1}}(M_{2})$};

\draw (h) node{$\bullet$};
\draw (i) node{$\bullet$};
\draw (1.5,-.4) node {$s_{0}(M)$};

\draw[fill=none](h) circle (.4) ;%node [black,yshift=-1.5cm];
\draw (.79,1.4) node {$O$};
\draw (-.35,.5) node {$\Gamma$};
\draw (f)--(g);
\draw (h)--(d);
\draw (d)--(i);

\coordinate (j) at (-.56,.35); % point M2
\coordinate (k) at (2.15,-.45); % point M1
\coordinate (l) at (0.36,0.08);
\draw (l) node{$\bullet$};
\draw (.18,-0.05) node {$t_{\rho_{1}}(M)$};
\draw (l)--(i);
\end{tikzpicture}
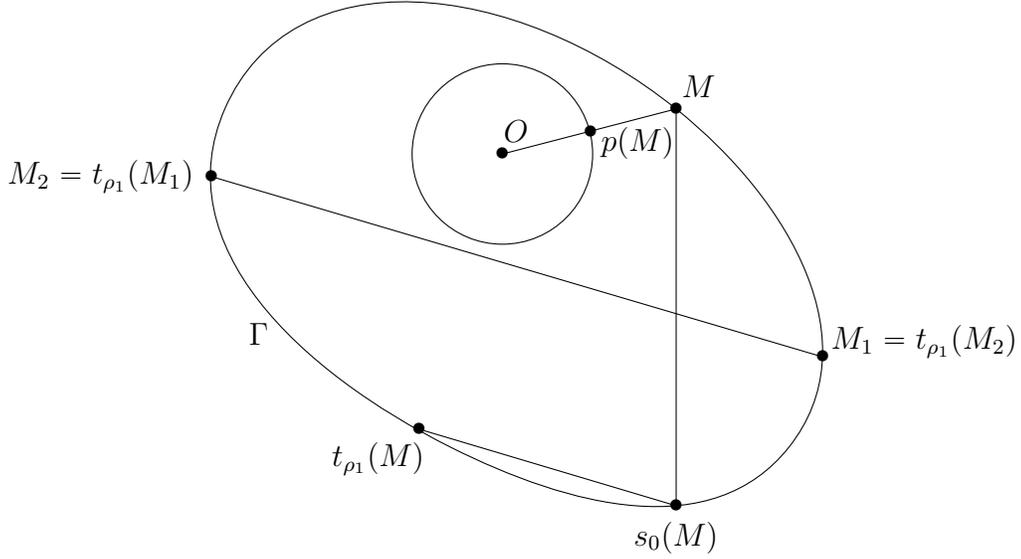
\\Figure 1 : The convex curve $\Gamma$, the circle $\T$, the points $M, M_{1}, M_{2}$ on $\Gamma$, \\
and some of their transforms, as defined in the text.
\end{center}
Moreover, let 
$$\tilde t_{\rho}:=p\circ t_{\rho}\circ p^{-1}:\T\to\T,\qquad\tilde g=g\circ p^{-1}\in L^{1}(\T),$$
where we notice that $\tilde t_{\rho}$ is a $C^{2}$ diffeomorphism of $\T$. Indeed, $\Gamma$ being a curve of class $C^{2}$, all of the bijective maps $s_{0}, s_{\rho}, t_{\rho}, p$ are diffeomorphisms of class $C^{2}$.

First, it is clear that, when $\rho=0$, $t_{0}$ is the identity map of $\Gamma$, and $\tilde t_{0}$ the identity map of $\T$, with rotation number 0. Second, consider the unique points $M_{1},M_{2}$ alluded to above, e.g.\ $M_{1}$ with the largest $x_{1}$-coordinate on $\Gamma$, and let choose $\rho=\rho_{1}$ such that the segment $(M_{1},M_{2})$ makes an angle $\rho_{1}+\pi/2$ with the oriented horizontal axis. Then, 
$$
t_{\rho_{1}}(M_{1})=s_{\rho_{1}}(M_{1})=M_{2}\quad\text{and}\quad t_{\rho_{1}}(M_{2})=s_{\rho_{1}}(M_{2})=M_{1}.
$$
Hence, $M_{1}$ is a 2-periodic point of $t_{\rho_{1}}$, and, similarly, $p(M_{1})$ is a 2-periodic point of $\tilde t_{\rho_{1}}$, which implies that the rotation number of $\tilde t_{\rho_{1}}$ is $1/2$. From the continuity of the map $\rho\mapsto \tau(\tilde t_{\rho})$, recall Lemma \ref{cont-rot}, we derive that there exists some angle $0<\rho_{2}<\rho_{1}$ such that $\tau_{2}:=\tau(\tilde t_{\rho_{2}})\not\in\Q$. From Denjoy's theorem \ref{Denjoy}, there exists a homeomorphism $h:\T\to\T$ such that $h\circ\tilde t_{\rho_{2}}\circ h^{-1}$ is the rotation $R_{\tau_{2}}$ of angle 
$2\pi\tau_{2}$. It then follows, with (\ref{inv-g-gamma}), that
$$
(\tilde g\circ h^{-1})(x)=(\tilde g\circ h^{-1})(R_{\tau_{2}}(x)),\qquad x\in\T.
$$
By the fact that $\tau_{2}$ is irrational, and the uniqueness of the Fourier coefficients of $\tilde g\circ h^{-1}$, we obtain as above that the function $\tilde g\circ h^{-1}$ is constant on $\T$, hence also the function $g$ on $\Gamma$. From the first equality in (\ref{inv-g-gam}), we finally conclude that $g=0$ almost everywhere on~$\Gamma$. Applying property (\ref{rem-T}) with a rotation finishes the proof of Theorem \ref{HUP2}.
%\begin{theorem}
%Let $\Gamma$ be any strictly convex planar curve. There exists an angle $\rho$ such that the pair 
%$(\Gamma,\Delta_{0}\cup\Delta_{\rho})$ is a HUP for $L^{1}(\Gamma)$. 
%\end{theorem}
%%%%%%%%%%%%%%%%%
\section{Appendix}
We give a proof of Theorem \ref{single-layer-cont} about the continuity of single layer potential with weakly singular kernels.
\begin{proof}[Proof of Theorem \ref{single-layer-cont}]
First, when $x_{0}\not\in\Gamma$, $\Phi g$ is continuous at $x_{0}$ since
\\[5pt]
%par continuit\'e d'une int\'egrale avec param\`etre :\\
-- the map $x\mapsto \vphi(x-y)g(y)$ is continuous at $x_{0}$, for all $y\in\Gamma$,
\\[5pt]
-- $|\vphi(x-y)g(y)|\leq Mg(y)\in L^{p}(\Gamma)\subset L^{1}(\Gamma)$, where $M$ is the sup of $|\vphi(x-y)|$ when $x$ is in a neighborhood of $x_{0}$ that does not intersect $\Gamma$, and $y\in\Gamma$.
%est born\'e sur $V_{x_{0}}\times\Gamma$, \underline{est major\'e par $Mg(y)\in L^{1}(\Gamma)$}, o\`u $V_{x_{0}}$ est un voisinage de $x_{0}$ qui ne rencontre pas~$\Gamma$, donc major\'e comme fonction de $y$ par une fonction int\'egrable.
%\\[5pt]
%{\bf 2ieme cas} : 
\\[5pt]
\indent Second, assume $x_{0}\in\Gamma$. 
%For simplicity, we also assume $n=2$. The proof for $n>2$ is identical.
Let $\rm T\Gamma_{x_{0}}$ and $n_{x_{0}}$ be the $(n-1)$-dimensional tangeant space to $\Gamma$ at $x_{0}$, and the normal unit vector to $\Gamma$ at $x_{0}$. We write $y=x_{0}+\tau_{y}+\eta_{y}n_{x_{0}}\in\Gamma$ with $\tau\in \rm T\Gamma_{x_{0}}$, and we let $\eta=\gamma(\tau)$ be the equation of $\Gamma$ in a neighborhood of $x_{0}$. 
For some small enough $\delta>0$, let
$$\Gamma(x_{0},\delta)=\{y\in\Gamma,~\tau_{y}\in B(x_{0},\delta)\},$$
%o\`u $\delta$ est assez petit pour que $\Gamma(x_{0},\delta)$ soit parametr\'e par $\gamma$. 
where $B(x_{0},\delta)$ is the ball of radius $\delta$ in $\rm T\Gamma_{x_{0}}$. Finally, for $x$ sufficiently close to $x_{0}$, we write
$$
x=x_{0}+\tau_{x}+\eta_{x}n_{x_{0}}.
$$
%(pour $x$ assez pr\`es de $x_{0}$, on aura $-\delta\leq\xi_{x}\leq\delta$).
Then, we have
\begin{multline}\label{diff-Phi}
\Phi g(x)-\Phi g(x_{0})=\int_{y\in\Gamma(x_{0},\delta)}\vphi(x-y)g(y)d\sigma(y)-\int_{y\in\Gamma(x_{0},\delta)}\vphi(x_{0}-y)g(y)d\sigma(y)\\[5pt]
+\int_{y\in\Gamma\setminus\Gamma(x_{0},\delta)}(\vphi(x-y)-\vphi(x_{0}-y))g(y)d\sigma(y).
\end{multline}
For the first two integrals, for $x$ possibly equal to $x_{0}$, we have, by H\"older inequality, with $q$ the exponent conjugate to $p$,
\begin{multline*}
\left|\int_{y\in\Gamma(x_{0},\delta)}\vphi(x-y)g(y)d\sigma(y)\right|\\
\leq\left(\int_{y\in\Gamma(x_{0},\delta)}|\vphi(x-y)|^{q}d\sigma(y)\right)^{1/q}
\left(\int_{y\in\Gamma(x_{0},\delta)}|g(y)|^{p}d\sigma(y)\right)^{1/p}.
\end{multline*}
The last integral is finite since $g\in L^{p}_{d\sigma}(\Gamma)$. For the second one, note that
$$
|\vphi(x-y)|\leq\frac{C}{|x-y|^{\nu}}\leq\frac{C}{|\tau_{x}-\tau_{y}|^{\nu}}.
$$
Moreover, 
$$
d\sigma(y)^{2}=d\tau^{2}+d\eta^{2}=(1+\gamma'(t)^{2})d\tau^{2}\leq4d\tau^{2}
$$
for $\delta$ small (recall that $\gamma'$ is continuous and $\gamma'(0)=0$). Hence
$$
\left|\int_{y\in\Gamma(x_{0},\delta)}|\vphi(x-y)|^{q}d\sigma(y)\right|\leq 2C\int_{B(x_{0},\delta)}
\frac{d\tau}{|\tau_{x}-\tau|^{\nu q}},
$$
(with $\tau_{x}=0$ if $x=x_{0}$) where the last integral over the $(n-1)$-dimensional ball $B(x_{0},\delta)$ is convergent since
$$ 1+\nu/(n-1-\nu)<p\quad\iff\quad \nu q<n-1.
$$
Hence, by taking $\delta$ small, it can be made as small as we wish (uniformly in $\tau_{x}$).

It remains to show that the third integral in (\ref{diff-Phi}) can be made small. For $x$ close to $x_{0}$ so that $x\notin\Gamma\setminus\Gamma(x_{0},\delta)$, the function
$$
x\mapsto\int_{y\in\Gamma\setminus\Gamma(x_{0},\delta)}\vphi(x-y)g(y)d\sigma(y)
$$
is continuous at $x_{0}$ (same argument as in the first case), and thus, the third integral is small when $x$ is sufficiently close to $x_{0}$.
\end{proof}

\obeylines
\texttt{
%\medskip
%\medskip
St\'ephane Rigat, stephane.rigat@univ-amu.fr
Franck Wielonsky, franck.wielonsky@univ-amu.fr
\medskip
Aix Marseille Univ, CNRS, I2M, Marseille, France
%Laboratoire I2M - UMR CNRS 7373
%Universit\'e Aix-Marseille
%CMI 39 Rue Joliot Curie
%F-13453 Marseille Cedex 20, FRANCE
}


\begin{thebibliography}{99}
\bibitem{A} A.C.L. Ashton, On the rigorous foundations of the Fokas method for linear elliptic partial differential equations. Proc. R. Soc. Lond. Ser. A 468 (2012), 1325-1331.

\bibitem{BA} S. Bagchi, Heisenberg uniqueness pairs corresponding to a finite number of parallel lines. Adv. Math. 325 (2018), 814-823.

%\bibitem{B} H. Bauer, Probability theory, De Gruyter Studies in Mathematics 23. Walter de Gruyter and Co., Berlin, 1996.

\bibitem{BB} D. Blasi Babot, Heisenberg uniqueness pairs in the plane. Three parallel lines. Proc. Amer. Math. Soc. 141 (2013), 3899-3904.

%\bibitem{BC} W. Bertiger, C. Cosner, 
%Systems of second-order equations with nonnegative characteristic form. 
%Comm. Partial Differential Equations 4 (1979), 701-737. 

\bibitem{CH} T. Cazenave, A. Haraux, An Introduction to Semilinear Evolution Equations, Oxford University Press, 2006.

\bibitem{C} E.M. Chirka, Complex analytic sets, Mathematics and its Applications (Soviet Series), 46. Kluwer Academic Publishers Group, Dordrecht, 1989.

\bibitem{CK} D. Colton, R. Kress, Integral equation methods in scattering theory.
 Classics in Applied Mathematics, 72. Society for Industrial and Applied Mathematics (SIAM), Philadelphia, PA, 2013.

\bibitem{DL} R. Dautray, J-L. Lions, Mathematical analysis and numerical methods for science and technology. Vol. 1. Physical origins and classical methods. Springer-Verlag, Berlin, 1990.

\bibitem{DZ1} D.R. Dunninger, E.C. Zachmanoglou, 
The condition for uniqueness of solutions of the Dirichlet problem for the wave equation in coordinate rectangles. J. Math. Anal. Appl. 20 (1967), 17-21. 

\bibitem{DZ2} D.R. Dunninger, E.C. Zachmanoglou, The condition for uniqueness of the Dirichlet problem for hyperbolic equations in cylindrical domains. J. Math. Mech. 18, 1969, 763-766.

\bibitem{E} L.C. Evans, Partial differential equations. Second edition. Graduate Studies in Mathematics, 19. American Mathematical Society, Providence, RI, 2010.

%\bibitem{GT} D. Gilbarg, N.S. Trudinger, Elliptic partial differential equations of second order. Second edition. Grundlehren der Mathematischen Wissenschaften 224. Springer-Verlag, Berlin, 1983.

\bibitem{GS} D. Giri, R.K. Srivastava, Heisenberg uniqueness pairs for some algebraic curves in the plane. Adv. Math. 310 (2017), 993-1016.

\bibitem{GV} F. Gonzalez Vieli, A uniqueness result for the Fourier transform of measures on the sphere. Bull. Aust. Math. Soc. 86 (2012), 78-82.

\bibitem{GR} R.C. Gunning, H. Rossi, Analytic functions of several complex variables. Prentice-Hall, Inc., Englewood Cliffs, N.J. 1965.

\bibitem{HJ} V. Havin, B. J\"oricke, The uncertainty principle in harmonic analysis. Ergebnisse der Mathematik und ihrer Grenzgebiete 28. Springer-Verlag, Berlin, 1994.

\bibitem{HM} H. Hedenmalm, A. Montes-Rodriguez, Heisenberg uniqueness pairs and the Klein-Gordon equation. Ann. of Math. 173 (2011), 1507-1527.

\bibitem{H} L. H\"ormander, The analysis of linear partial differential operators. I. Distribution theory and Fourier analysis, Second edition, Springer-Verlag, Berlin, 1990.

\bibitem{JK} P. Jaming, K. Kellay, A dynamical system approach to Heisenberg uniqueness pairs. J. Anal. Math. 134 (2018), 273-301.

\bibitem{KH} A. Katok, B. Hasselblatt, Introduction to the modern theory of dynamical systems. 
%With a supplementary chapter by Katok and Leonardo Mendoza. 
Encyclopedia of Mathematics and its Applications, 54. Cambridge University Press, Cambridge, 1995.

\bibitem{L} N. Lev, Uniqueness theorems for Fourier transforms. Bull.\ Sci.\ Math.\ 135 (2011), 134-140.

%\bibitem{M} W. McLean,
%Strongly elliptic systems and boundary integral equations. Cambridge University Press, Cambridge, 2000.

%\bibitem{OR} O. A. Oleinik and E. V. Radkevic, Second order equations with nonnegative characteristic form, Plenum, London, 1973.

\bibitem{O} F.W.J. Olver, Asymptotics and special functions. Computer Science and Applied Mathematics. Academic Press, New York-London, 1974.

\bibitem{O2} F.W.J. Olver,
Error bounds for stationary phase approximations. 
SIAM J. Math. Anal. 5 (1974), 19-29.  

\bibitem{OW} N. Ortner, P. Wagner, Fundamental solutions of linear partial differential operators. 
Theory and practice. Springer, Cham, 2015.

%\bibitem{SH} B.V. Shabat, 
%Introduction to complex analysis. Part II. Functions of several variables. Translations of Mathematical Monographs, 110. AMS, Providence, RI, 1992.

%\bibitem{S} Y-T Siu, Hilbert Nullstellensatz in global complex-analytic case. Proc. Amer. Math. Soc. 19 (1968), 296-298.

\bibitem{SJ1} P. Sjolin, Heisenberg uniqueness pairs and a theorem of Beurling and Malliavin. Bull. Sci. Math. 135 (2011), 125-133.

\bibitem{SJ2} P. Sjolin, Heisenberg uniqueness pairs for the parabola. J. Fourier Anal. Appl. 19 (2013), 410-416.

\bibitem{SO} F. Sottile, 
Real algebraic geometry for geometric constraints. Handbook of geometric constraint systems principles, 273-285, Discrete Math. Appl. (Boca Raton), CRC Press, Boca Raton, FL, 2019.

\bibitem{SR} R.K. Srivastava, Non-harmonic cones are Heisenberg uniqueness pairs for the Fourier transform on $\R^{n}$. J. Fourier Anal. Appl. 24 (2018), 1425-1437.

\bibitem{T} F. Tr\`eves, Basic linear partial differential equations. Pure and Applied Mathematics, Vol. 62. Academic Press, New York-London, 1975.
\end{thebibliography}
\end{document}